\documentclass{amsart}
\usepackage{amsmath,amssymb,amsfonts}
\usepackage{graphicx}

\newif{\ifshowchanges}
\showchangestrue

\ifshowchanges

\newcommand{\del}[1]{{\bf [(DELETE) #1]}}
\else

\newcommand{\del}[1]{\relax}
\fi

\newcommand{\A}{\mathcal A}
\newcommand{\R}{\mathbb R}
\newcommand{\Z}{\mathbb Z}
\newcommand{\ZZ}{Dias}

\newcommand{\T}{\mathcal T}
\newcommand{\onen}{1,\dots,n}
\newcommand{\ioner}{{i_1\dots i_r}}

\newcommand{\eps}{\varepsilon}

\newcommand{\GO}{\Omega}
\newcommand{\Gt}{\theta}
\newcommand{\Ls}{L^\sigma}
\newcommand{\mb}{\bar{\mu}}
\newcommand{\tp}{\tilde\pi}
\newcommand{\tF}{\widetilde{F}}
\newcommand{\tZ}{\widetilde{Z}}

\newtheorem{thm}{Theorem}[section]

\newtheorem{prop}[thm]{Proposition}
\newtheorem{cor}[thm]{Corollary}
\theoremstyle{definition}
\newtheorem{ex}[thm]{Example}
\theoremstyle{definition}
\newtheorem{defn}[thm]{Definition}
\theoremstyle{remark}
\newtheorem{rem}[thm]{Remark}
\newcommand{\lk}{\operatorname{lk}}
\def\sgn{\operatorname{sign}}
\def\sminus{\smallsetminus}
\def\<{\langle}
\def\>{\rangle}

\def\A{\mathcal{A}}

\def\Dias{{\it Dias}}
\def\lkfigright{\begin{picture}(30,20)(-5,5)\put(0,0){\line(0,1){20}}
   \put(20,0){\line(0,1){20}}\put(0,10){\vector(1,0){20}}\end{picture}}
\def\lkfigleft{\begin{picture}(30,20)(-5,5)\put(0,0){\line(0,1){20}}
   \put(20,0){\line(0,1){20}}\put(20,10){\vector(-1,0){20}}\end{picture}}
\def\midfig#1#2{\begin{picture}(50,30)(-5,10)\put(0,0){\line(0,1){30}}
   \put(20,0){\line(0,1){30}}\put(40,0){\line(0,1){30}}#1#2\end{picture}}
\def\bigfig#1#2#3{\begin{picture}(70,40)(-5,15)\put(0,0){\line(0,1){40}}
   \put(20,0){\line(0,1){40}}\put(40,0){\line(0,1){40}}
   \put(60,0){\line(0,1){40}}#1#2#3\end{picture}}
\def\onetwo#1{\put(0,#10){\vector(1,0){20}}}
\def\onethree#1{\put(0,#10){\vector(1,0){40}}}

\def\twoone#1{\put(20,#10){\vector(-1,0){20}}}
\def\twothree#1{\put(20,#10){\vector(1,0){20}}}
\def\twofour#1{\put(20,#10){\vector(1,0){40}}}
\def\threeone#1{\put(40,#10){\vector(-1,0){40}}}
\def\threetwo#1{\put(40,#10){\vector(-1,0){20}}}
\def\threefour#1{\put(40,#10){\vector(1,0){20}}}
\def\fourone#1{\put(60,#10){\vector(-1,0){60}}}
\def\fourtwo#1{\put(60,#10){\vector(-1,0){40}}}
\def\fourthree#1{\put(60,#10){\vector(-1,0){20}}}
\begin{document}

\title{Diassociative algebras and Milnor's invariants for tangles}
\author{Olga Kravchenko}
\address{Universit\'e de Lyon, Universit\'e Lyon 1, ICJ, UMR 5208 CNRS, 43 blvd 11 novembre 1918, 69622 Villeurbanne CEDEX, France }
\email{okra@math.univ-lyon1.fr}
\author{Michael Polyak}
\address{Department of mathematics, Technion, Haifa 32000, Israel}
\email{polyak@math.technion.ac.il}

\begin{abstract}
We extend Milnor's $\mu$-invariants of link homotopy to ordered (classical
or virtual) tangles. Simple combinatorial formulas for $\mu$-invariants are
given in terms of counting trees in Gauss diagrams. Invariance under
Reidemeister moves corresponds to axioms of Loday's diassociative algebra.
The relation of tangles to diassociative algebras is
formulated in terms of a morphism of corresponding operads.
\end{abstract}

\thanks{The second author was partially supported by the ISF grant 1343/10}
\subjclass[2010]{57M25; 57M27; 18D50; 16S37}
\keywords{tangles, $\mu$-invariants, planar trees, dialgebras, operads}

\maketitle

\section{Introduction}
The theory of links studies embeddings of several disjoint copies of
$S^1$ into $\R^3$ and thus has to deal with a mixture of linking and
self-knotting phenomena. The theory of link-homotopy was developed
by Milnor \cite{Mi}  in order to isolate the linking phenomena from
the self-knotting ones and to study linking separately. A fundamental
set of link-homotopy invariants is given by Milnor's $\mb_{\ioner,j}$
invariants \cite{Mi} with non-repeating indices $1\le i_1,\dots i_r,j\le n$.
Roughly speaking, these describe the dependence of the $j$-th
parallel on the meridians of the $i_1,\dots,i_r$ components.
The simplest invariant $\mb_{i,j}$ is just the linking number of the
corresponding components. The next one, $\mb_{i_1i_2,j}$, detects
Borromean-type linking of the corresponding 3 components and
together with the linking numbers classify 3-component links up
to link-homotopy.

There is no semi-group structure defined on multi-component links such
as one existing for knots. Namely, connected sum, while well-defined for
knots, is not defined for links. On the level of invariants, this is manifested
by a complicated recurrent indeterminacy in the definition of the
$\mb$-invariants (reflected in the use of notation $\mb$, rather than
$\mu$). Introduction of string links in \cite{HL} remedied this situation,
since connected sum is well-defined for string links. A version of
$\mb$-invariants modified for string links is thus free of the original
indeterminacy; to stress this fact, we use the notation $\mu$ for
these invariants from now on. Milnor's invariants classify string links
up to link-homotopy (\cite{HL}).

\subsection{Brief statement of results}
Tangles generalize links, braids and string links.
We define Milnor's $\mu-$invariants for tangles with ordered components
along the lines of Milnor's original definition, that is in terms of generators
of the (reduced) fundamental group of the complement of a tangle in a
cylinder, using the Magnus expansion.

On the other hand, tangles may be encoded by Gauss diagrams
(see \cite{PV,GPV}). We follow the philosophy of \cite{PV} to
define invariants of classical or virtual tangles by counting (with
appropriate weights and signs) certain subdiagrams of a Gauss diagram.
Since subdiagrams used in computing these invariants correspond to
rooted planar binary  trees, we call the resulting invariants $Z_j$
\textit{tree invariants}.

Invariance of tangle diagrams under Reidemeister moves gives rise
to several equivalence relations among the corresponding trees.
We study these relations and find (Theorem \ref{thm:invt}) that they
could be interpreted as defining relations of a diassociative algebra.
The notion of diassociative algebra was introduced by  Loday \cite{Lo}.
A diassociative algebra is a vector space with two associative
operations -- left and right multiplications. The five defining axioms
(equation \ref{eq:dias}) of diassociative algebra describe invariance
under the third Reidemeister move.

We  explicitly write out the linear combinations of trees used in computing
invariants of degrees 2,3 and 4. In particular, tree invariants $Z_{12,3}$
and $Z_{123,4}$ are computed and subsequently shown to coincide with
the corresponding Milnor $\mu-$invariants.

Then we discuss the properties of tree invariants of (classical or
virtual) tangles. In particular, we study their dependence on orderings
and orientations of strings. Moreover, we show that these invariants
satisfy certain skein relations, reminiscent of those satisfied by the
Conway polynomial and the Kauffman bracket. The skein relations
for Milnor invariants were determined by the second author in \cite{P1}.
Similarity of skein relations of tree invariants to Milnor's invariants
allows us to show that tree invariants $Z_{\ioner,j}$ coincide with
Milnor's $\mu$-invariants $\mu_{\ioner,j}$ when $1\le j< i_1<\dots <i_r\le n$.
This also allows us to extend Milnor's $\mu$-invariants to virtual tangles.

To describe the operadic structure on tangles we introduce the notion
of a tree tangle. For tree tangles there is an appropriate operation of
grafting, which allows us to define the operad of tree tangles.
We show that there is a map from tangles to tree tangles by an operation
called \textit{capping}. We describe a morphism of operads between the
operad of tree tangles and the diassociative algebra operad $\Dias$.

The paper is organized in the following way.
In Section \ref{sec:prelim} the main objects and tools are introduced:
tangles, Milnor's $\mu$-invariants, and Gauss diagram formulas.
In Section \ref{sec:invts} we review diassociative algebras and
introduce tree invariants of tangles and prove their invariance under
Reidemeister moves.
Section  \ref{sec:properties} is devoted to the properties of the invariants
and their identification with the $\mu$-invariants.
Finally, in Section  \ref{sec:operad} we discuss the operadic structure on
tree tangles and the corresponding morphism of operads.

The authors are grateful to Paul Bressler, Fr\'ed\'eric Chapoton
and Jean-Louis Loday for stimulating discussions, and to the
French consulate in Israel for a generous travel support.

\section{Preliminaries}\label{sec:prelim}

\subsection{Tangles and string links}
Let $D^2$ be the unit disk in $xy$-plane and let $p_i$, $i=1,\dots,N$
be some prescribed points in the interior of $D^2$. For definiteness,
we can chose the disk to have the center at $(1,0)$ and the points
lying on the $x$-axis.
\begin{defn} An (ordered, oriented) {\em $(k,l)$-tangle} without closed
components in the cylinder $C=D^2\times[0,1]$ is an ordered collection of
$n=\frac12(k+l)$ disjoint oriented intervals, properly embedded in $C$
in such a way, that the endpoints of each embedded interval belong to
the set $\{p_i\}_{i=1}^k\times\{1\}\cup\{p_i\}_{i=1}^l\times\{0\}$ in 
$C$. See Figure \ref{fig:tangles}a. We will call embedded intervals the 
{\em strings} of a tangle. Tangles are considered up to an oriented 
isotopy in $C$, fixed on the boundary.
\end{defn}
We will always assume that the only singularities of (the image of) the
projection of a tangle to the $xz$-plane are transversal double points.
Such a projection, equipped with the indication of over- and underpasses
in each double point, is called a {\em tangle diagram}.
See Figure \ref{fig:tangles}b.
\begin{figure}[htb]
\centerline{\includegraphics[width=4in]{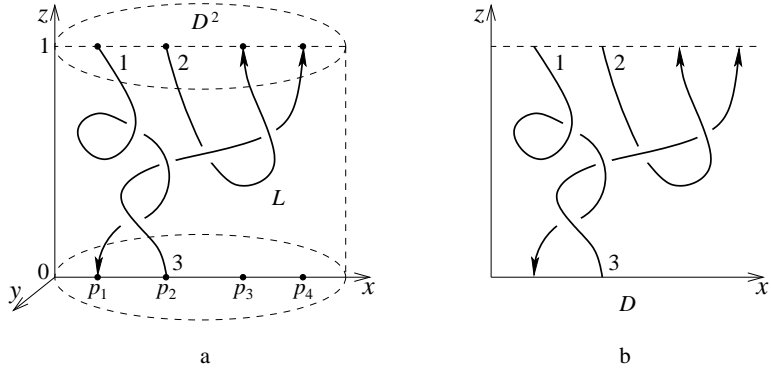}}
\caption{\label{fig:tangles} A $(4,2)$-tangle and its diagram}
\end{figure}

{\em String links } form an important class of tangles which is
comprised by $(n,n)$-tangles such that the $i$-th arc ends in
the points $p_i\times\{0,1\}$, see Figure \ref{fig:string}a.
By the {\em closure} $\widehat{L}$ of a string link $L$ we mean
the braid closure of $L$. It is an $n$-component link obtained
from $L$ by an addition of $n$ disjoint arcs in the $xz$-plane,
each of which meets $C$ only at the endpoints $p_i\times\{0,1\}$
of $L$, as illustrated in Figure \ref{fig:string}b. The linking
number $\lk$ of two strings of $L$ is their linking number in
$\widehat{L}$.
\begin{figure}[htb]
\centerline{\includegraphics[width=5in]{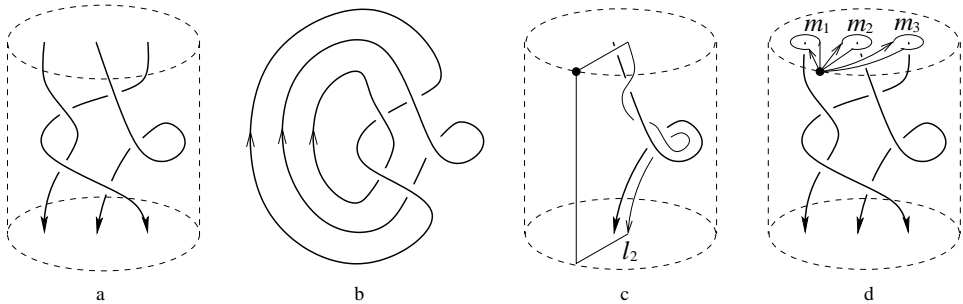}}
\caption{\label{fig:string} A string link, its closure, and canonical meridians and parallels}
\end{figure}
Two tangles are {\em link-homotopic}, if one can be transformed
into the other by homotopy, which fails to be isotopy only in a
finite number of instants, when a (generic) self-intersection point
appears on one of the arcs.

\subsection{Milnor's $\mu$-invariants}\label{sub:mu}
Let us briefly recall the construction of Milnor's link-homotopy
$\mu$-invariants (see \cite{Mi} for details, \cite{Le} for a
modification to string links, and \cite{P1} for the case of tangles).
We will first describe the well-studied case of string links,
and then indicate modifications needed for the general
case of tangles.

Let $L=\cup_{i=1}^nL_i$ be an $n$-component string link and
consider the link group $\pi=\pi_1(C\smallsetminus L)$ with the
base point $(1,1,1)$ on the upper boundary disc $D^2\times\{1\}$.
Choose {\em canonical parallels} $l_j\in\pi$, $j=\onen$ represented
by curves going parallel to $L_j$ and then closed up by standard
non-intersecting curves on the boundary of $C$ so that $\lk(l_j,L_j)=0$;
see Figure \ref{fig:string}c.
Also, denote by $m_i\in\pi$, $i=\onen$ the {\em canonical meridians}
represented by the standard non-intersecting curves in $D^2\times\{1\}$
with $\lk(m_i,L_i)=+1$, as shown in Figure \ref{fig:string}d. If $L$
is a braid, these meridians freely generate $\pi$, with any other
meridian of $L_i$ in $\pi$ being a conjugate of $m_i$. For general
string links, similar results hold for the reduced link group $\tp$.

Given a finitely-generated group $G$, the {\em reduced group}
$\tilde{G}$ is the quotient of $G$ by relations $[g,w^{-1}gw]=1$,
for any $g, w\in G$. One can show (see \cite{HL}) that $\tp$
is generated by $m_i$, $i=\onen$ proceeding similarly to the
usual construction of Wirtinger's presentation.
Let $F$ be the free group on $n$ generators $x_1,\dots x_n$.
The map $F\to\pi$ defined by $x_i\mapsto m_i$ induces the
isomorphism $\tF\cong\tp$ of the reduced groups \cite{HL}.
We will use the same notation for the elements of $\pi$ and
their images in $\tp\cong\tF$.

Now, let $\Z[[X_1,\dots,X_n]]$ be the ring of power series in $n$
non-commuting variables $X_i$ and denote by $\tZ$ its quotient by
the two-sided ideal generated by all  monomials, in which at least
one of the generators appears more than once.
The {\em Magnus expansion} is a ring homomorphism of the
group ring $\Z F$ into $\Z[[X_1,\dots,X_n]]$, defined by
$x_i\mapsto 1+ X_i, \ x_i^{-1} \mapsto 1 - X_i + X_i^2 - \cdots$.
It induces the homomorphism $\Gt:\Z\tF\to\tZ$ of the corresponding
reduced group rings.
In particular, for the case of $\tF$ being the link group of a link $L$
there is the homomorphism of reduced group rings  $\Gt_L:\Z\tp\to\tZ$.

{\em Milnor's invariants} $\mu_{\ioner,j}(L)$ of the string link $L$
are defined as coefficients of the Magnus expansion $\Gt_L(l_j)$ of
the parallel $l_j$:
$$\Gt_L(l_j)=\sum\mu_{\ioner,j}X_{i_1}X_{i_2}\dots X_{i_r}\ .$$
In particular, if $L_j$ passes everywhere in front of the other
components, all the invariants $\mu_{\ioner,j}$ vanish. Modulo lower
degree invariants $\mu_{\ioner,j}(L)\equiv\mb_{\ioner,j}(\widehat{L})$,
where $\mb_{\ioner,j}(\widehat{L})$ are the original Milnor's link
invariants \cite{Mi}.

The above definition of invariants $\mu_{\ioner,j}(L)$ may be
adapted to ordered oriented tangles without closed components
in a straightforward way. The canonical meridian $m_i$ of $L_i$
is defined as a standard curve on the boundary of $C$, making a
small loop around the starting point of $L_i$ (with $\lk(m_i,L_i)=+1$).
A canonical parallel $l_j$ of $L_j$ is a standard closure of a pushed-off
copy of $L_j$ (with $\lk(l_j,L_j)=0$). See Figure  \ref{fig:m-l}.
The only difference with the string link case is that for general tangles
there is no well-defined canonical closure (some additional choices --
e.g. of a marked component -- are needed).
\begin{figure}[htb]
\centerline{\includegraphics[width=5in]{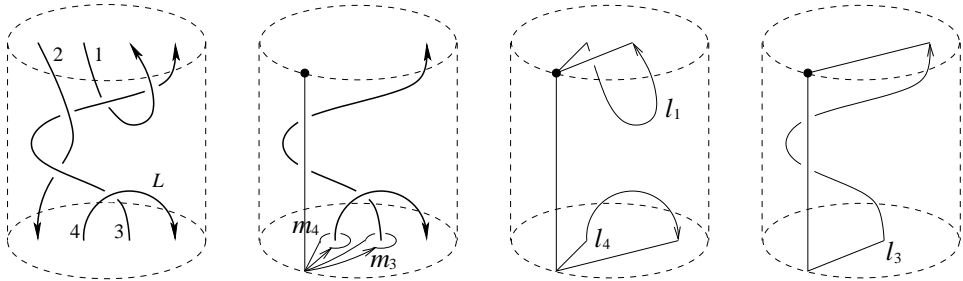}}
\caption{\label{fig:m-l} A choice of canonical meridians and parallels for a tangle}
\end{figure}
\begin{rem}\label{rem:order}
Note that the invariants $\mu_{\ioner,j}$ significantly depend on
the order of indices $i_1,i_2,\dots,i_r$ and $j$ (e.g., in general
$\mu_{i_1i_2\dots i_r,j}(L)\ne\mu_{i_2i_1\dots i_r,j}(L)$).
Under a permutation $\sigma\in S_n$, $\sigma:i\mapsto\sigma(i)$
$\mu$-invariants change in an obvious way: $\mu_{i_1i_2\dots
i_r,j}(L')=\mu_{\sigma(i_1)\sigma(i_2)\dots\sigma(i_r),\sigma(j)}(L)$,
where $L'$ is the tangle $L$ with changed ordering: $L'_i=L_{\sigma(i)}$.
\end{rem}

\subsection{Gauss diagrams}
Gauss diagrams provide a simple combinatorial way to encode links
and tangles. Consider a tangle diagram $D$ as an immersion 
$D:\sqcup_{i=1}^n I_i\to\R^2$ of $n$ disjoint copies of the
unit interval into the $xz$-plane, equipped with information 
about the overpass and the underpass in each crossing.
\begin{defn}\label{defn:Gauss}
Let $L$ be a $(k,l)$-tangle and $D$ its diagram.
The {\em Gauss diagram} $G$ corresponding to  $D$  is an ordered
collection of $n = \frac12(k+l)$ intervals $\sqcup_{i=1}^n I_i$ with 
the preimages of each crossing of $D$ connected by an arrow.
Arrows are pointing from the over-passing string to the under-passing
string and are equipped with the sign: $\pm1$ of the corresponding
crossing (its local writhe).
\end{defn}
We will usually depict the intervals in a Gauss diagram as vertical lines,
assuming that they are oriented downwards and ordered from left to right.
See Figure \ref{fig:gauss}.
\begin{figure}[htb]
\centerline{\includegraphics[width=4.4in]{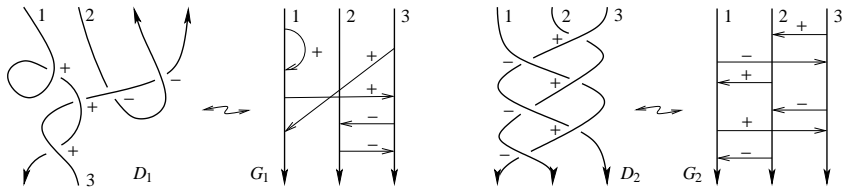}}
\caption{\label{fig:gauss} Gauss diagrams}
\end{figure}

The Gauss diagram of a tangle, $G$, encodes all the information
about the crossings, and thus all the essential information contained
in the tangle diagram $D$, in a sense that, given endpoints of each
string, $D$ can be reconstructed from $G$ uniquely up to isotopy.
Reidemeister moves of tangle diagrams may be easily translated
into the language of Gauss diagrams, see Figure \ref{fig:Reidem}.
Here fragments participating in a move may be parts of the same
string or belong to different strings, ordered in an arbitrary fashion,
and the fragments in $\GO1$ and $\GO2$ may have different
orientations. It suffices to consider only one oriented move of type
three, see \cite{CDBook,P3}.
\begin{figure}[htb]
\centerline{\includegraphics[width=4.4in]{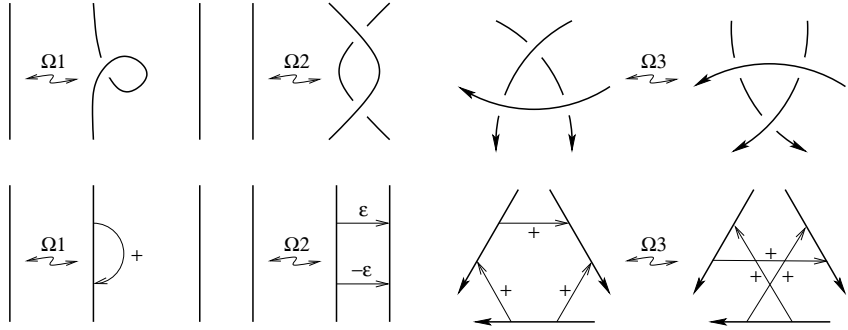}}
\caption{\label{fig:Reidem} Reidemeister moves for diagrams and Gauss diagrams}
\end{figure}

\subsection{Virtual tangles}\label{sub:virtual}
Note that not all collections of arrows connecting a set of $n$
strings can be realized as a Gauss diagram of some tangle.
Dropping this realization requirement leads to the theory of virtual
tangles, see \cite{Ka, GPV}. We may simply define a virtual
tangle as an equivalence class of virtual (that is, not necessary
realizable) Gauss diagrams modulo the Reidemeister moves of
Figure \ref{fig:Reidem}.

The fundamental group $\pi_1(C\sminus L)$ may be explicitly
deduced from a Gauss diagram of a tangle $L$. It is easy to check
that the fundamental group is invariant under the Reidemeister
moves. Thus, the construction of Section \ref{sub:mu} may be
carried out for virtual tangles as well, resulting in a definition of
$\mu$-invariants of virtual tangles.

The only new feature in the virtual case is the existence of two
tangle groups. This is related to a possibility to choose the
base point for the computation of the fundamental group
$\pi=\pi_1(C\sminus L)$ either in the front half-space $y>0$
(see Figure \ref{fig:string} and Section \ref{sub:mu}), or in the
back half-space $y<0$. While for classical tangles Wirtinger
presentations obtained using one of these base points are two
different presentations of the same group $\pi$, for virtual
tangles we get two different - the upper and the lower -
tangle groups. See \cite{GPV} for details. The passage from the
upper to the lower group corresponds to a reversal of directions
(but not of signs!) of all arrows in a Gauss diagram.
Using the lower group in the construction of Section \ref{sub:mu},
we would end up with another definition of $\mu$-invariants,
leading to a different set of ``lower $\mu$-invariants'' in the
virtual case.
We will return to this discussion in Remark \ref{rem:virtual} below.

\subsection{Gauss diagram formulas}
\begin{defn}\label{defn:arrowdiag} An {\it arrow diagram on $n$ strings}
is an ordered set of $n$ oriented intervals (strings), with several arrows
connecting pairs of distinct points on intervals, considered up to orientation
preserving diffeomorphism of the intervals.
\end{defn}
See Figure \ref{fig:arrow}. In other words, an  arrow diagram is a virtual
Gauss diagram in which we forget about realizability and signs of arrows.
\begin{figure}[htb]
\centerline{\includegraphics[width=3.8in]{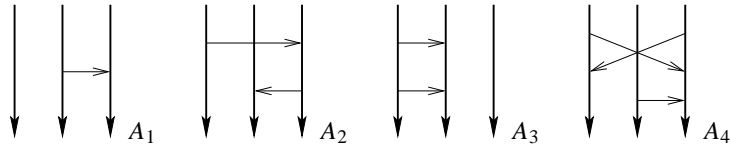}}
\caption{\label{fig:arrow} Arrow diagrams}
\end{figure}

Given an arrow diagram $A$ on $n$ strings and a Gauss diagram
$G$ with $n$ intervals, we define a map $\phi:A\to G$ as an
embedding of $A$ into $G$ which maps intervals to intervals and
arrows to arrows, preserving their orientations and ordering of intervals.
The sign of $\phi$ is defined as $\sgn(\phi)=\prod_{a\in A}\sgn(\phi(a))$.
Finally, define a pairing $\< A, G\>$ as
$$\< A, G\>=\sum_{\phi:A\to G}\sgn(\phi)$$
and if there is no embedding of $A \to G,$ then $\< A, G\>= 0.$
For example, for arrow diagrams $A_1,A_2,A_3,A_4$ of Figure
\ref{fig:arrow} and Gauss diagrams $G_1$, $G_2$ shown in
Figure \ref{fig:gauss},  we have
$\< A_1, G_1\>=\< A_2, G_1\>=\< A_4, G_1\>=-1$, $\<A_2, G_2\>=1$
and $\<A_3, G_1\>=\< A_1, G_2\>=\< A_3, G_2\>=\< A_4, G_2\>=0$.
We extend $\< \,\cdot\, , G\>$ to a vector space generated by all arrow
diagrams on $n$ strings by linearity.

For some special linear combinations $A$ of arrow diagrams the
expression $\<A,G\>$ is  preserved under the Reidemeister moves
of $G$, thus resulting in an invariant of (ordered) tangles.
See \cite{PV} and \cite{GPV} for details and a general discussion
on this type of formulas. The simplest example of such an invariant
is a well-known formula for the linking number of two components:
\begin{equation}\label{lk_eq}
\lk(L_1,L_2)=\<\lkfigleft,G\>.
\end{equation}

The right hand side is the sum $\sum_{\phi:A\to G}\sgn(\phi)$ over
all maps of $A=\lkfigleft$ to $G$. In other words, it is just the sum of
signs of all crossings of $D$, where $L_1$ passes under $L_2$.

\begin{rem}
Note that for string links one has
$$\lk(L_1,L_2)=\<\lkfigleft,G\>=\<\lkfigright,G\>=\lk(L_2,L_1).$$
For general tangles, however, these two invariants may differ.
For example, for a tangle diagram with just one crossing, where
$L_1$ passes in front of $L_2$, we have $\<\lkfigleft,G\>=0$ and
$\<\lkfigright,G\>=\pm 1$ depending on the sign of the crossing.
This is a simple illustration of a general phenomenon: symmetries,
which usually hold for classical links and string links, break down
for tangles and virtual links. We will return to this observation in
Section \ref{sec:invts}.
\end{rem}

In the next section we introduce Gauss diagram formulas for a
family of tangle invariants which includes all Milnor's link-homotopy
$\mu$-invariants.

\section{Tangle invariants by counting trees}\label{sec:invts}
In what follows, let $I=\{i_1,i_2\dots,i_r\}$, $1\le i_1<i_2<\dots< i_r\le n$
and $j\in\{1,2,\dots,n\}\smallsetminus I$.

\subsection{Tree diagrams}
\begin{defn}
A {\em tree diagram} $A$ with leaves on strings numbered by $I$
and a trunk on $j$-th string is an arrow diagram which satisfies the
following conditions:
\begin{itemize}
\item An arrowtail and an arrowhead of an arrow belong to different
strings;
\item There is exactly one arrow with an arrowtail on $i$-th string,
if $i\in I$, and no such arrows if $i\notin I$;
\item All arrows have arrowheads on $I\cup\{j\}$ strings;
\item All arrowheads precede the (unique) arrowtail for each $i \in I$,
as we follow the $i$-th strand string in the sense of its orientation.
\end{itemize}
\end{defn}

Note that the total number of arrows in a tree diagram is
$r=|I|$; we will call this number the {\em degree} of $A$.
Our choice of the term tree diagram is explained by the following.
Consider $A$ as a graph (with vertices being heads and tails
of arrows and beginning and ending points of the strings).
Removing all $k$-strings where  $k \notin I\cup\{j\}$, and cutting off
the part of each of the remaining strings after the corresponding
arrowtail, we obtain a tree $T_A$ with $r+1$ leaves on the
beginning of each $i$-string with $i \in I\cup\{j\}$ and the root in
the endpoint of $j$-th string. We will also say that $T_A$ is a
tree with leaves on $I$ and a trunk on $j$.
See Figure \ref{fig:tree}, where some tree diagrams with $r=2$,
$j=1$, $I=\{2,3\}$ are shown together with corresponding trees.

\begin{figure}[htb]
\centerline{\includegraphics[width=5in]{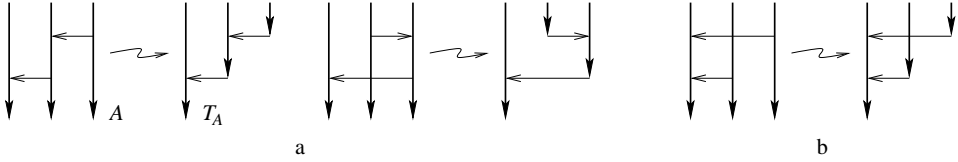}}
\caption{Planar and non-planar tree diagrams} \label{fig:tree}
\end{figure}

Note that every tree $T_A$ could be realized as a planar graph.
The tree diagram $A$ is called {\em planar}, if in its planar
realization the order of the leaves coincides with the initial
ordering $i_1<i_2<\dots<i_l<j<i_{l+1}<\dots<i_r$ of the strings
as we count the leaves starting from the root clockwise.
For example, diagrams in Figure \ref{fig:tree}a are planar, while the
one in Figure \ref{fig:tree}b is not. Let $\A_{I,j}$ denote the set of
all planar tree diagrams with leaves on $I$ and a trunk on
$j$ and let $\A_j=\displaystyle{\cup_I\A_{I,j}}$.

\subsection{Diassociative algebras and trees}
Let the sign of an arrow diagram $A$ be $\sgn(A)=(-1)^q$,
where $q$ is the number of right-pointing arrows in $A$.
Given a Gauss diagram $G$ of a tangle with the marked $j$-th
string, we define the following quantity, taking value in a free
abelian group generated by planar rooted trees\footnote{Note that
this sum is always finite, since the Gauss diagram contains a fixed
number of strings.}:
$$\sum_{A\in\A_j}\sgn(A)\<A,G\>\cdot T_A$$
While this formal sum of trees fails to be a tangle invariant, it
becomes one modulo certain equivalence relations on trees.
These relations turn out to be the axioms of a diassociative
algebra (also known as associative dialgebra):

\begin{defn}(\cite{Lo})
A diassociative algebra over a ground field $k$ is a $k$-space $V$
equipped with two $k$-linear maps
\[
\vdash:\ V\otimes V\to V\quad\mbox{and}
 \quad\dashv:\ V\otimes V\to V,
\]
called left and right products and satisfying the following five
axioms:
\begin{equation}\label{eq:dias}
\left\{ \begin{array}{cc}
(1)& (x\dashv y) \dashv z = x \dashv (y \vdash z) \\
(2) & (x\dashv y) \dashv z = x \dashv (y \dashv z)  \\
(3) & (x\vdash y) \dashv z = x \vdash (y \dashv z) \\
(4)& (x\dashv y) \vdash z = x \vdash (y \vdash z) \\
(5)& (x\vdash y) \vdash z = x \vdash (y \vdash z)
\end{array}
\right.
\end{equation}
\end{defn}


Diagrammatically, one can think about a free diassociative algebra
as follows. Depict products $a\vdash b$ and $a\dashv b$ as
elementary trees shown in Figure \ref{fig:dias}a.
Composition of these operations corresponds then to grafting of
trees, see Figure \ref{fig:dias}b,c.

\begin{figure}[htb]
\centerline{\includegraphics[width=5in]{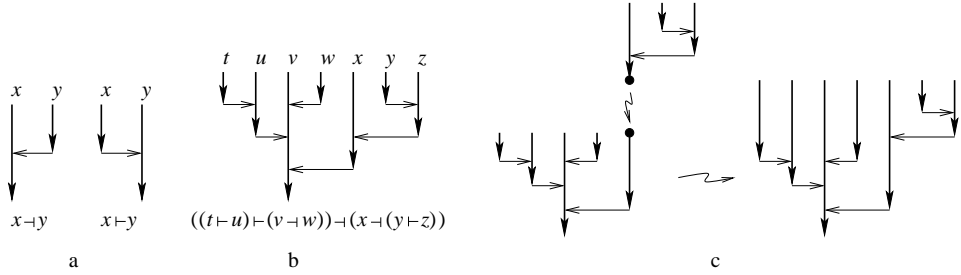}}
\caption{Diassociative operations as trees and their compositions}
\label{fig:dias}
\end{figure}

Axioms \eqref{eq:dias} correspond to relations on trees shown in
Figure \ref{fig:relsdias}.

\begin{figure}[htb]
\centerline{\includegraphics[width=5in]{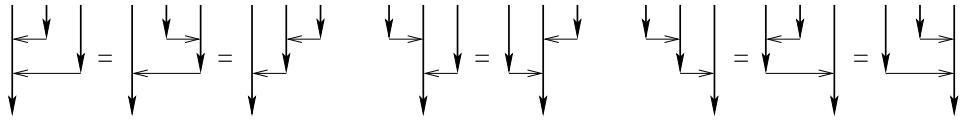}}
\caption{Diassociative algebra relations on trees}
\label{fig:relsdias}
\end{figure}

Denote by $\ZZ(n)$ the quotient of the vector space generated
by planar rooted trees with $n$ leaves by the axioms of the
diassociative algebra and let $\ZZ=\cup_n\ZZ(n)$.
The operadic composition on  $\ZZ$ corresponds to grafting of
trees, as illustrated in Figure \ref{fig:dias}c. See \cite{Lo} for details.


\subsection{Tree invariants}
\label{sub:dias}
Let $[T]$ denote the equivalence class of a planar tree $T$ in $\ZZ$, and
$G$ be the Gauss diagram of a tangle. Then $Z_j(G)\in\ZZ$ is defined as
\begin{equation}
\label{eq:Z}
Z_j(G)=\sum_{A\in\A_j}\sgn(A)\<A,G\>[T_A]
\end{equation}
$T_A$ being the tree corresponding to the tree diagram $A$.
We call $Z_j(G)$ the \textit{tree invariant} of a tangle which has
$G$ as its Gauss diagram, since it satisfies the following

\begin{thm}\label{thm:invt}
Let $L$ be an ordered (classical or virtual) tangle and let $G$ be a
Gauss diagram of $L$. Then $Z_{j}(L)=Z_j(G)$ is an
invariant of ordered tangles.
\end{thm}

\begin{proof}
It suffices to prove that $Z_j(G)$ is preserved under the Reidemeister
moves $\GO1$--$\GO3$ for Gauss diagrams (Figure \ref{fig:Reidem}).
Given a Gauss diagram $G$, invariance of $Z_j(G)$ under $\GO1$
and $\GO2$ follows immediately from the definition of tree diagrams.
Indeed, a new arrow appearing in $\GO1$ has both its arrowhead
and its arrowtail on the same string, so it cannot be in the image of a
tree diagram $A$.
Hence the \eqref{eq:Z} rests intact under the first move.
It is also invariant under the second move for the following reason.
Two new arrows which appear in $\GO2$ have their arrowtails on the
same string, so they cannot simultaneously belong to the image
of a tree diagram, while maps which contain one of them cancel out
in pairs due to opposite signs of the two arrows.

It remains to verify invariance under the third Reidemeister move
$\GO3$ depicted in Figure \ref{fig:Reidem}. Denote by $G$ and
$G'$ Gauss diagrams related by $\GO3$. Note that there is a
bijective correspondence between the summands of $Z_j(G)$
and those of $Z_j(G')$. Indeed, since only the relative position of
the three arrows participating in the move changes, all terms which
involve only one of these arrows do not change. No terms involve
all three arrows, since such a diagram cannot be a tree diagram.
It remains to compare terms which involve exactly two arrows.
Note that a diagram which involves two arrows can be a tree
diagram only if the fragments participating in the move belong to
three different strings. There is a number of cases, depending on
the ordering $\sigma_1,\sigma_2,\sigma_3$ of these three
strings. Using for simplicity indices $1,2,3$ for such an ordering,
we can summarize the correspondence of these terms in the table below.

\vspace{0.1in}
\centerline{\includegraphics[width=5in]{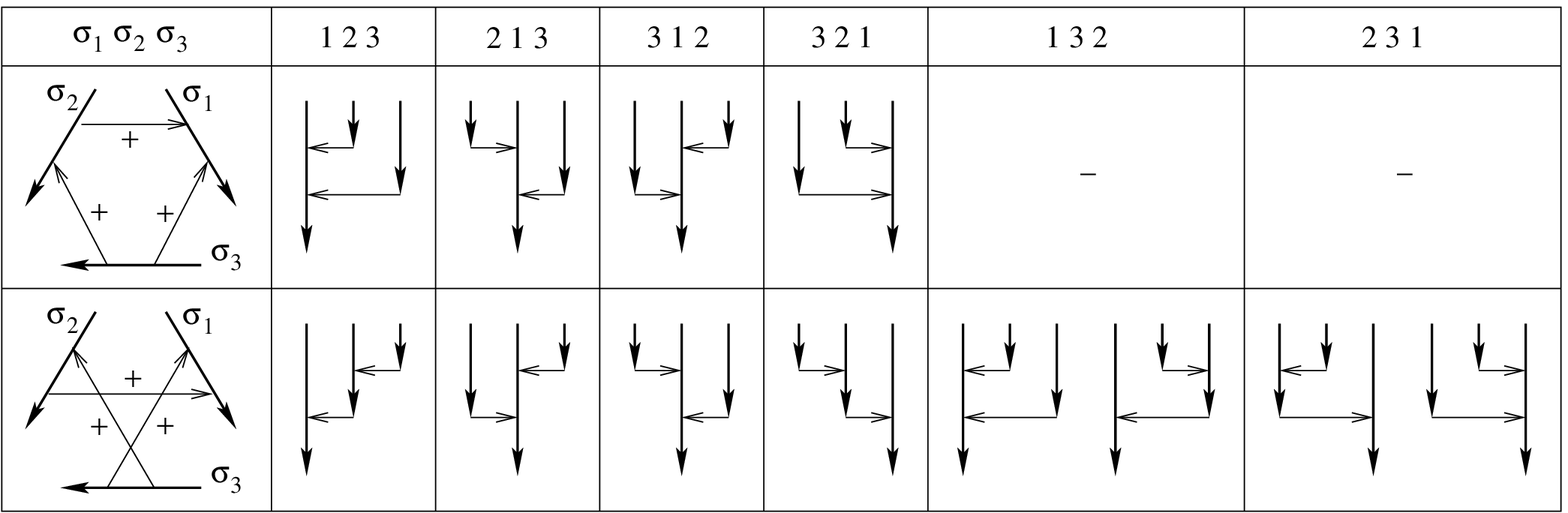}}
\vspace{0.1in}

We see that invariance is assured exactly by the diassociative
algebra relations, see Figure \ref{fig:relsdias}.
For four orderings out of six the correspondence is bijective, while
for the two last orderings, pairs of trees appearing in the bottom
row have opposite signs (due to different number of right-pointing
arrows), so their contributions to $Z_j(G')$ cancel out.
\end{proof}

\section{Properties of the tree invariants}\label{sec:properties}

The tree invariant $Z_j(L)$ takes values in the quotient $\ZZ$ of the
free abelian group generated by trees by the diassociative algebra
relations. The equivalence class $[T_A]$ of a tree $T_A$ with trunk
on $j$ depends only on the set of its leaves, so it is the same for
all arrow diagrams $A$ in the set $\A_{I,j}$ of all planar tree arrow
diagrams with leaves on $I$ and trunk on $j$.

Let $Z_{I,j}$ be the coefficient of $Z_j$
corresponding to trees with leaves on $I$, namely,
$Z_{I,j}=\sum_{A\in\A_{I,j}}\sgn(A)\<A,G\>$.
For $I=\emptyset$ we set $Z_{\emptyset,j}=1$.

\subsection{Invariants in low degrees} \label{subsec:lowdeg}
Let us start with invariants $Z_{I,j}$ for small values of $r=|I|$.

Counting tree diagrams with one arrow we get
\begin{equation}\label{eq:mu2}
Z_{2,1}(L)=\<\lkfigleft,G\>\ , \qquad Z_{1,2}(L)=-\<\lkfigright,G\>.
\end{equation}
Note that if $L$ is a string link $Z_{2,1}(L)=-Z_{1,2}(L)=\lk(L_1,L_2)$.

For diagrams with two arrows we obtain
\begin{multline}\label{eq:mu3}
Z_{23,1}(L)=\<\midfig{\threetwo2}{\twoone1}+\midfig{\twoone2}{\threeone1}
-\midfig{\twothree2}{\threeone1},G\>\ , \ Z_{13,2}(L)=
-\<\midfig{\onetwo2}{\threetwo1}+\midfig{\threetwo2}{\onetwo1},G\>\ , \\
Z_{12,3}(L)=\<\midfig{\onetwo2}{\twothree1}+\midfig{\twothree2}{\onethree1}
-\midfig{\twoone2}{\onethree1},G\>\hspace{2.1in}
\end{multline}
In particular, $Z_{13,2}(L)=Z_{1,2}(L)\cdot Z_{3,2}(L)$.
Also, $Z_{12,3}(L)=Z_{23,1}(\bar{L})$, where $\bar{L}$ is the
tangle $L$ with reflected ordering $\bar{L}_i=L_{4-i}$ of strings.


\begin{ex}\label{ex:borromean}
Consider a tangle $L$ with corresponding diagram $D_2$ depicted in
Figure \ref{fig:gauss} and let us compute $Z_{23,1}(L)$ using formula
\eqref{eq:mu3}. The corresponding Gauss diagram $G_2$ contains
three subdiagrams of the type $\midfig{\threetwo2}{\twoone1}$, two
of which cancel out, while the remaining one contributes $+1$; there
are no subdiagrams of other types appearing in \eqref{eq:mu3}.
Hence, $Z_{23,1}(L)=1$.
\end{ex}

When an orientation of a component is reversed, invariants $Z_{I,j}$
change sign and jump by a combination of lower degree invariants.
For example, denote by $L'$ the 3-string tangle obtained from $L$
by reversal of orientations of $L_1$. Then,
$$Z_{23,1}(L')=\<-\midfig{\threetwo2}{\twoone1}+\midfig{\threeone2}{\twoone1}
+\midfig{\twothree2}{\threeone1},G\>.$$
But it is easy to see that
$\<\midfig{\threeone2}{\twoone1}+\midfig{\twoone2}{\threeone1},G\>=
\<\midfig{\twoone2}{},G\>\cdot\<\midfig{\threeone1}{},G\>$,
thus we obtain
$$Z_{23,1}(L')=-Z_{23,1}(L)+Z_{2,1}(L)\cdot Z_{3,1}(L).$$

Let us write down explicitly 3-arrow diagrams with trunk on the first string:
\begin{multline}\label{eq:mu4}
Z_{234,1}(L)=\<\bigfig{\twoone3}{\threeone2}{\fourone1}
-\bigfig{\twoone2}{\threefour2}{\fourone1}
+\bigfig{\threefour3}{\twofour2}{\fourone1}
+\bigfig{\threetwo3}{\twoone2}{\fourone1}
-\bigfig{\threetwo3}{\twofour2}{\fourone1}
-\bigfig{\twothree3}{\threeone2}{\fourone1}\\
+\bigfig{\twoone2}{\fourthree2}{\threeone1}
-\bigfig{\twothree3}{\fourthree2}{\threeone1}
-\bigfig{\fourthree3}{\twothree2}{\threeone1}
+\bigfig{\twothree3}{\threefour2}{\fourone1}
+\bigfig{\threetwo3}{\fourtwo2}{\twoone1}
+\bigfig{\fourthree3}{\threetwo2}{\twoone1}
-\bigfig{\threefour3}{\fourtwo2}{\twoone1},G\>
\end{multline}
\vspace{0.1in}
For diagrams with trunk on the second or third strings we have
$Z_{134,2}(L)=Z_{1,2}(L)\cdot Z_{34,2}(L)$,
$Z_{124,3}(L)=Z_{12,3}(L)\cdot Z_{4,3}(L)$. Finally, for $j=4$ we
have $Z_{123,4}(L)=-Z_{432,1}(\bar{L})$, where $\bar{L}$ is
obtained from $L$ by the reflection $\bar{L}_i=L_{5-i}$ of the ordering.

\subsection{Elementary properties of tree invariants}
Unlike $\mu$-invariants discussed in Section \ref{sub:mu} which
had simple behavior under change of ordering (see Remark
\ref{rem:order}), tree invariants $Z_{I,j}(L)$ depend significantly on
the order of $i_1,\dots,i_r$ and $j$. Namely, if $L'_{i}=L_{\sigma(i)}$
for some $\sigma\in S_n$, $\sigma:i\to\sigma(i)$, then, in general,
$Z_{I,j}(L')$ is not directly related to $Z_{\sigma(I),\sigma(j)}(L)$.
However, in some simple cases dependence of tree invariants on
ordering and their behavior under simple changes of ordering and
reflections of orientation can be deduced directly from their definition
via planar trees:
\begin{prop}\label{prop:properties}
Let $L$ be an ordered (classical or virtual) tangle on $n$ stringsand let $I=\{i_1,i_2,\dots,i_r\}$, with $1\le i_1<i_2<\dots<i_r\le n$.
\begin{enumerate}

\item
For $1<k<r$ we have
$$Z_{I\sminus i_k,i_k}(L)=Z_{I_k^-,i_k}(L)\cdot Z_{I_k^+,i_k}(L)$$
where $I_k^-=I\cap[1,i_k-1]=\{i_1,\dots,i_{k-1}\}$ and
$I_k^+=I\cap[i_k+1,n]=\{i_{k+1},\dots,i_r\}$.

\item
Denote by $\bar{L}$ the tangle $L$ with reflected ordering:
$\bar{L}_i=L_{\bar{i}}$, $i=1,\dots,n$, where $\bar{i}=n+1-i$,
so $\bar{I}=\{\bar{i_r},\dots,\bar{i_2},\bar{i_1}\}$. Then
$$Z_{I,j}(\bar{L})=(-1)^r Z_{\bar{I},\bar{j}}(L)$$

\item
Finally, denote by $\Ls$ the tangle the tangle obtained from $L$ by
cyclic permutation $\sigma=(i_1 i_2\dots i_r )$ of strings of $L$
(that is, $\Ls_{i_k}=L_{i_{k+1}}$ for $k=1,\dots,r-1$ and
$\Ls_{i_r}=L_{i_1}$), followed by the reversal of orientation
of the last string $\Ls_{i_r}=L_{i_1}$. Then
$$Z_{I\sminus i_r,i_r}(\Ls)=Z_{I\sminus i_1,i_1}(L)$$
\end{enumerate}
\end{prop}
\begin{proof}
Indeed, a planar tree with trunk on $j$ consists of the ``left
half-tree" with leaves on $I\cap[1,j-1]$ and the ``right half-tree"
with leaves on $I\cap[j+1,n]$. Thus the first equality follows
directly from the definition of the invariants.

Also, the reflection $i\mapsto \bar{i}$ of ordering simply reflects a
planar tree with respect to its trunk, exchanging the left and the
right half-trees and changing all right-pointing arrows into
left-pointing ones and vice versa, so the second equality follows
(since the total number of arrows is $r$).

Finally, let us compare planar tree subdiagrams in the Gauss diagram $G$
of $L$ and in the corresponding Gauss diagram $G^\sigma$ of $\Ls$.
Cyclic permutation $\sigma$ of ordering, followed by the reversal of
orientation of the trunk, establishes a bijective correspondence between
planar tree diagrams with leaves on $I\sminus i_1$ and trunk on $i_1$
and planar tree diagrams with leaves on $I\sminus i_r$ and trunk on $i_r$.
Given a diagram $A\in\A_{i_1}$, we can obtain the corresponding diagram
$A^\sigma \in\A_{i_r}$ in two steps: (1) redraw the trunk $i_1$ of $A$ on
the right of all strings, with an upwards orientation; (2) reverse the orientation
of the trunk so that it is directed downwards. See Figure \ref{fig:shift}.
\begin{figure}[htb]
\centerline{\includegraphics[width=4.6in]{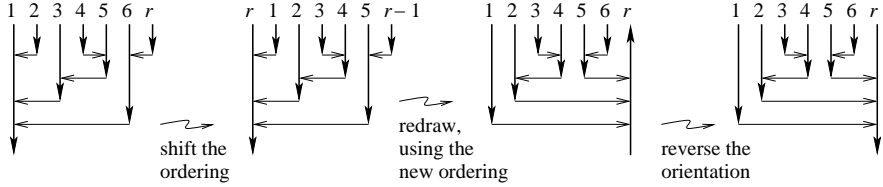}}
\caption{Reordering strings and reversing the orientation of the trunk}
\label{fig:shift}
\end{figure}
Signs of these diagrams are related as follows:
$\sgn(A)=(-1)^q \sgn(A^\sigma)$, where $q$ is the number of arrows
with arrowheads on the trunk (since all such arrows become
right-pointing instead of left-pointining).
Now note that when we pass from $G$ to $G^\sigma$, the reflection
of orientation of $\Ls_{i_r}$ has similar effect on the signs of arrows,
namely, the sign of each arrow in $G^\sigma$ with one end on the trunk
(and the other end on some other string) is reversed, so
$\<A,G\>=(-1)^q \<A^\sigma,G^\sigma\>$. These two factors of $(-1)^q$
cancel out to give $\sgn(A)\<A,G\>=\sgn(A^\sigma)\<A^\sigma,G^\sigma\>$
and the last statement follows.
\end{proof}

Tree invariants $Z_{I,j}(L)$ satisfy the following skein relations.
Let $L_+$, $L_-$, $L_0$ and $L_\infty$ be four tangles which differ
only in the neighborhood of a single crossing $d$, where they look as
shown in Figure \ref{fig:skein}. In other words, $L_+$ has a
positive crossing, $L_-$ has a negative crossing, $L_0$ is obtained
from $L_\pm$ by smoothing, and $L_\infty$ is obtained from $L_\pm$
by the reflection of orientation of $L_{i_k}$, followed by
smoothing. Orders of strings of $L_\pm$, $L_0$ and $L_\infty$
coincide in the beginning of each string. See Figures
\ref{fig:skein} and \ref{fig:example}. We will call $L_\pm$, $L_0$
and $L_\infty$ a {\em skein quadruple}.
\begin{figure}[htb]
\centerline{\includegraphics[width=3.6in]{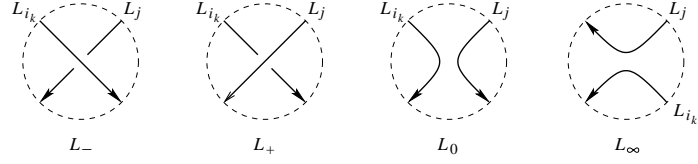}}
\caption{Skein quadruple of tangles}
\label{fig:skein}
\end{figure}
\begin{thm}\label{thm:skein}
Let $j<i_1<i_2<\dots <i_r$ and $1\le k\le r$.
Let $L_+$, $L_-$, $L_0$ and $L_\infty$ be a skein quadruple of
tangles on $n$ strings which differ only in the neighborhood of a
single crossing $d$ of $j$-th and $i_k$-th components, see Figure
\ref{fig:skein}.
For $m=1,\dots,k$ denote $I_m^-=\{i_1,\dots,i_{m-1}\}$,
$I_m^+=I\sminus I_m^-\sminus i_k=\{i_m,\dots,i_{k-1},i_{k+1},\dots,i_r\}$.
Then
\begin{equation}\label{eq:skein_mu}
Z_{I,j}(L_+)-Z_{I,j}(L_-)=Z_{I_k^-,j}(L_\infty)\cdot
Z_{I_k^+,i_k}(L_0) \ ;
\end{equation}
\begin{equation}\label{eq:skein}
Z_{I,j}(L_+)-Z_{I,j}(L_-)=\sum_{m=1}^{k}Z_{I_m^-,j}(L_\pm)\cdot
Z_{I_m^+,i_k}(L_0) \ .
\end{equation}
Here we used the notation $Z_{I_m^-,j}(L_\pm)$ to stress that
$Z_{I_m^-,j}(L_+)=Z_{I_m^-,j}(L_-)$.
\end{thm}

\begin{rem}
Note that for $m=1$ we have $I_1^-=\emptyset$ and 
$I_1^+=I\sminus i_k$, which corresponds to the summand 
$Z_{I\sminus i_k,i_k}(L_0)$ in the right hand side of 
\eqref{eq:skein}. Also, in the particular case $k=1$ both of the
equations \eqref{eq:skein_mu},\eqref{eq:skein} simplify to
\begin{equation}
\label{eq:keq1}
Z_{I,j}(L_+)-Z_{I,j}(L_-)=Z_{I\sminus i_1,i_1}(L_0) \qquad (k=1)
\end{equation}
Finally, for $k=r$ equation \eqref{eq:skein_mu} becomes
$$Z_{I,j}(L_+)-Z_{I,j}(L_-)=Z_{I\sminus i_r,j}(L_\infty) \qquad (k=r)$$
\end{rem}

\begin{ex}\label{ex:borromean_skein}
Consider the tangle $L=L_+$ depicted in Figure \ref{fig:example} and
let us compute $Z_{23,1}(L)$. Notice that if we switch the indicated
crossing of $L_1$ with $L_2$ to the negative one, we get the link $L_-$
with $L_3$ unlinked from $L_1$ and $L_2$, so $Z_{23,1}(L_-)=0$.
We have $i_1=2,i_2=3$ and $k=1$, thus we can use equation
\eqref{eq:keq1} and get
$$Z_{23,1}(L)=Z_{23,1}(L)-Z_{23,1}(L_-)=Z_{3,2}(L_0)=1,$$
in agreement with the calculations of Example \ref{ex:borromean}.
\end{ex}
\begin{figure}[htb]
\centerline{\includegraphics[height=1.1in]{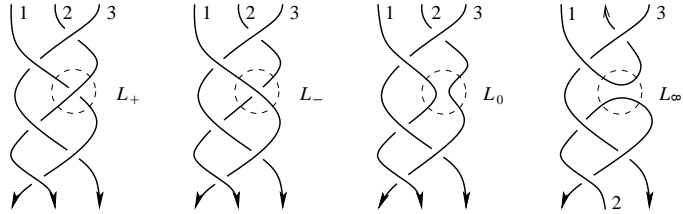}}
\caption{\label{fig:example} Computation of $Z_{23,1}$ for Borromean
rings}
\end{figure}

\begin{proof} To prove Theorem \ref{thm:skein} consider Gauss
diagrams $G_\eps$ of $L_\eps$, $\eps=\pm$ in a neighborhood
of the arrow $a_\pm$ corresponding to the crossing $d$ of $L_\pm$,
see Figure \ref{fig:skeinG_pm0infty}a.

\begin{figure}[htb]
\centerline{\includegraphics[width=5in]{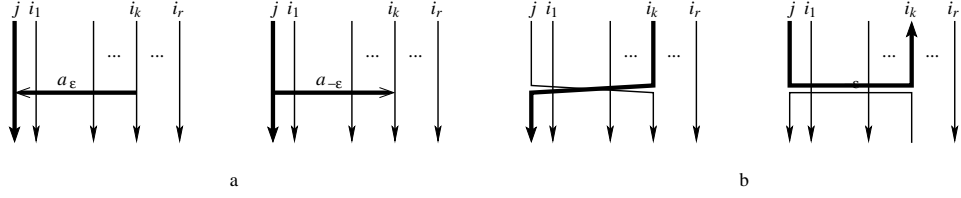}}
\caption{Gauss diagrams which appear in skein relations}
\label{fig:skeinG_pm0infty}
\end{figure}

Here if $L_j$ passes under $L_{i_k}$ in the crossing $d$ of $L_+$
$\eps=+$, and $\eps=-$ otherwise. There is an obvious bijective
correspondence between tree subdiagrams of $G_+$ and $G_-$
which do not include $a_\pm$, so these subdiagrams cancel out in
pairs in $\<A,G_+\>-\<A,G_-\>$.
Since we count only trees with the root on $j$-th string, the only
subdiagrams which contribute to $Z_{I,j}(L_+)-Z_{I,j}(L_-)$ are
subdiagrams of $G_+$ which contain $a_+$ if $\eps=+$, and
subdiagrams of $G_-$ which contain $a_-$ if $\eps=-$.
Note that in each case the arrow $a_\pm$ is counted with the
positive sign (since if $\eps=-1$, it appears in $-Z_{I,j}(L_-)$).
Without loss of generality we may assume that $\eps=+$.
Thus, $$Z_{I,j}(L_+)-Z_{I,j}(L_-)=\sum_{A\in\A_{I,j}}\<A,G_+\>_{a_+}\ ,$$
where $\<A,G\>_a$ denote the sum of all maps $\phi:A\to G$ such that
$a\in \text{Im}(\phi)$. See the left hand side of Figure \ref{fig:skeinGauss}.

\begin{figure}[htb]
\centerline{\includegraphics[height=1.5in]{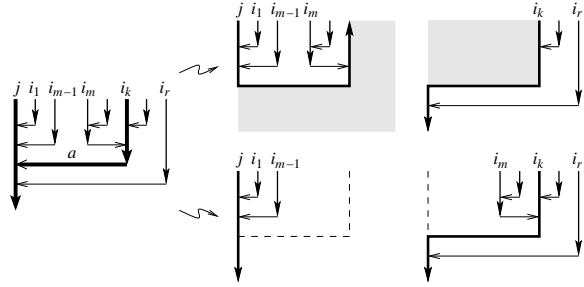}}
\caption{Skein relations on Gauss diagrams} \label{fig:skeinGauss}
\end{figure}

Interpreting $L_0$ and $L_\infty$ in terms of Gauss diagrams as
shown in Figure \ref{fig:skeinG_pm0infty}b, and using Proposition
\ref{prop:properties}, we immediately get equality \eqref{eq:skein_mu}.
See the top row of Figure \ref{fig:skeinGauss}.

Subdiagrams which appear in the equality \eqref{eq:skein} are shown
in the bottom row of Figure \ref{fig:skeinGauss}. To establish
\eqref{eq:skein}, it remains to understand why subdiagrams which
contain arrows with arrowheads on $j$ under $a_+$ cancel out in
$\sum_{m=1}^{k}Z_{I_m^-,j}(L_\pm)\cdot Z_{I_m^+,i_k}(L_0)$.
Fix $1\le m\le k$ and let $A_1\in\A_{I_m^-,j}$ and $A_2\in\A_{I_m^+,i_k}$
be two tree arrow diagrams together with maps $\phi_1:A_1\to G_+$,
$\phi_2:A_2\to G_0$. Suppose that one of the subdiagrams $G_1=\text{Im}(\phi_1(A_1))$ and $G_2=\text{Im}(\phi_2(A_2))$ of
$G_+$ contains an arrow, which ends on $j$-th string under $a$.
Denote by $a_{bot}$ the lowest such arrow in $G_1\cup G_2$ (as
we follow $j$-th string along the orientation). Without loss of generality,
we may assume that it belongs to $G_1$. See Figure \ref{fig:skein_cancel}.
\begin{figure}[htb]
\centerline{\includegraphics[width=5in]{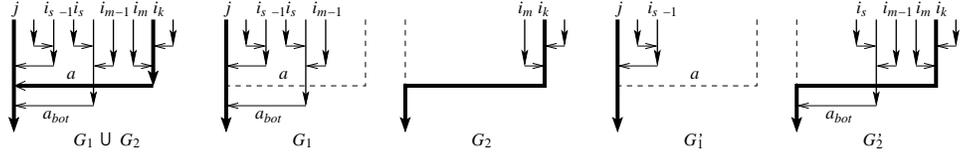}}
\caption{Cancelation of subdiagrams with arrows under $a$}
\label{fig:skein_cancel}
\end{figure}
Since $a_{bot}$ ends on the common part of the trunks of $G_+$ and $G_0$,
we may rearrange pieces of $G_1$ to get two different tree diagrams
with the same set of arrows as $G_1\cup G_2$.
Namely, removal of $a_{bot}$ from $G_1$ splits it into two connected
components $G_1'$ and $G_1''$, so that $G_1'$ contains strings
$j, i_1,\dots, i_{s-1}$ and $G_1''$ contains strings $i_s,\dots,i_{m-1}$
for some $1\le s\le m$.
Then $G_1'$ is a tree subdiagram of $G_+$ (with trunk on $j$ and
leaves on $I_{s}^-$), and $G_2':=G_1''\cup a_{bot}\cup G_2$ is a tree
subdiagram of $G_0$ (with the trunk on $i_k$ and leaves on $I_s^+$).
See Figure \ref{fig:skein_cancel}.
Their contribution to $Z_{I_s^-,j}(L_\pm)\cdot Z_{I_s^+,i_k}(L_0)$ cancels out
with that of $G_1$ and $G_2$ to $Z_{I_m^-,j}(L_\pm)\cdot Z_{I_m^+,i_k}(L_0)$.
Indeed, while $G_1'\cup G_2'$ contain the same set of arrows as $G_1\cup G_2$,
the arrow $a_{bot}$ is now right-pointing, so it is counted with additional
factor of $-1$. This completes the proof of the theorem.
\end{proof}

\subsection{Identification with Milnor's $\mu$-invariants}
\label{subsec:identific}
It turns out, that for $j < i, \ \forall i \in I$, the tree
invariant $Z_{I,j}$ coincides with a Milnor's $\mu$-invariant:
\begin{thm}\label{thm:mu}
Let $L$ be an ordered (classical or virtual) tangle on $n$ strings and let
$1\le j< i_1<i_2<\dots<i_r\le n$.  Then
$$Z_{I,j}(L)=\mu_{\ioner,j}(L)$$
\end{thm}

\begin{proof}
Theorem 3.1 of \cite{P1} (together with Remark \ref{rem:order})
implies that $\mu_{\ioner,j}(L)$ satisfies the same skein relation as
\eqref{eq:skein_mu}, that is
$$\mu_{I,j}(L_+)-\mu_{I,j}(L_-)=\mu_{I_k^-,j}(L_\infty)\cdot
\mu_{I_k^+,i_k}(L_0)\ .$$
Moreover, these invariants have the same normalization
$Z_{I,j}(L)=\mu_{I,j}(L)=0$ for any tangle $L$ with the $j$-th string
passing  in front of all other strings.
The skein relation and the normalization completely
determines the invariant.
\end{proof}
\begin{cor}
 Formulas \eqref{eq:mu3} and \eqref{eq:mu4} define invariants
$\mu_{23,1}$
and $\mu_{234,1}$ respectively.
\end{cor}

\begin{ex}
 If we return to the tangle $L$ of Examples \ref{ex:borromean} and
\ref{ex:borromean_skein}, shown in Figure \ref{fig:example}, we get
$\mu_{23,1}(L)=Z_{23,1}=1$, in agreement with the fact that the
closure $\widehat{L}$ of $L$ is the Borromean link.
\end{ex}

\begin{rem}\label{rem:virtual}
Note that in the proof of Theorem \ref{thm:invt} we did not use the
realizability of Gauss diagrams in our verification of invariance of
tree invariants under Reidemeister moves in  Figure \ref{fig:Reidem},
so Theorems \ref{thm:invt} and \ref{thm:mu} hold for virtual tangles
as well. Recall, however, that in the virtual case there is an alternative
definition of "lower" $\mu$-invariants of virtual tangles via the
lower tangle group, see Section \ref{sub:virtual}. To recover these
invariants using Gauss diagram formulas we simply reverse directions
of all arrows in the definition of the set of tree diagrams $\A_j$.
\end{rem}

\section{Operadic structure of the invariants}\label{sec:operad}

\subsection{Tree tangles}
\begin{defn}
A tree tangle $L$ is a $(k,1)$-tangle without closed components. The
string ending on the bottom (that is, on $D^2 \times \{0\}$) is called the
trunk of $L$.
\end{defn}
We will assume that tree tangles are oriented in such a way that
the trunk starts at the top $D^2 \times \{1\}$ and ends on the
bottom $D^2 \times \{0\}$ of $C$.
To simplify the notation, for a tree tangle $L$ with the trunk on the
$j$-th string we will denote $Z_j(L)$ by  $Z(L)$.
There is a natural way to associate to a $(k,l)$-tangle with a
distinguished string a tree tangle by pulling up all but one
of its strings. Namely, suppose  that the $j$-th string of a
$(k,l)$-tangle $L$ starts at the top and ends on the bottom.
Then $L$ can be made into a tree $(k+l-1,1)$-tangle $\widehat{L}_j$
with the trunk on $j$-th string by the operation of $j-$capping shown
in Figure \ref{fig:capping}.

\begin{figure}[htb]
\centerline{\includegraphics[width=5in]{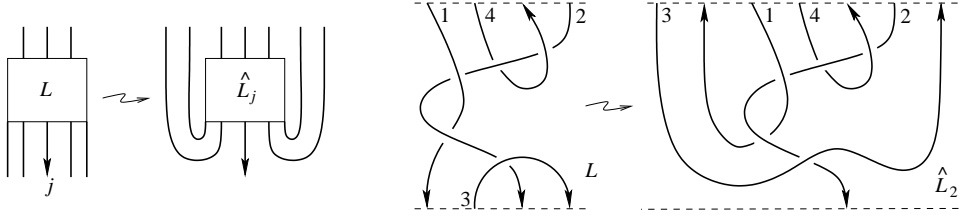}}
\caption{Capping a tangle} \label{fig:capping}
\end{figure}

Gauss diagrams of $L$ and $\widehat{L}_j$ are the same (since
crossings of $\widehat{L}_j$ are the same as in $L$), so their
tree invariants coincide: $Z_j(L) = Z(\widehat{L}_j)$.

\subsection{Operadic structure on tree tangles}
Denote by $\T(n)$ the set of tree tangles on $n$ strings. Tree
tangles form an operad $\T$. The operadic composition
\[
\T(n) \times \T(m_1) \times \cdots \times \T(m_n) \to \T(m_1+
\cdots + m_n)
\]
is defined as follows. A partial composition
$\circ_i:\T(n)\times\T(m)\to \T(n+m-1)$ corresponds to taking the
satellite of the $i$-th component of a tangle:
\begin{defn}
Let $L\in\T(n)$ and $L'\in\T(m)$ be tree tangles, and let $1\le i\le n$.
Define the satellite tangle $L\circ_i L'\in\T(n+m-1)$ as follows.
Cut out of $C=D^2\times[0,1]$ a tubular neighborhood $N(L_i)$ of the
$i$-th string $L_i$ of $L$. Glue back into
$C\smallsetminus N(L_i)$ a copy of a cylinder
$C$ which contains $L'$, identifying the boundary
$\partial D^2\times[0,1]$ with the boundary of $N(L_i)$ in
$C\smallsetminus N(L_i)$ using the zero
framing\footnote{In fact, the result does not depend on the framing
since only one component of $L'$ ends on the bottom of the
cylinder.} of $L_i$. See Figure \ref{fig:satellite}. Reorder components
of the resulting tree tangle appropriately.
\end{defn}

\begin{figure}[htb]
\centerline{\includegraphics[width=5in]{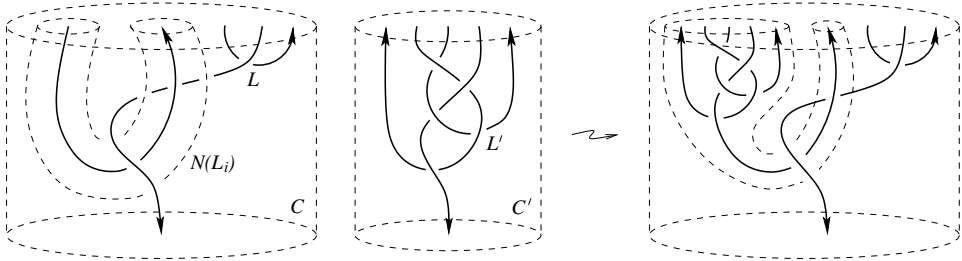}}
\caption{The satellite $L\circ_i L'$ of the $i$-th string of the tree tangle $L$} \label{fig:satellite}
\end{figure}

Now, given a tangle $L\in\T(n)$ and a collection of $n$ tree tangles
$L^1\in \T(m_1)$,\dots, $L^n\in \T(m_n)$, we define the composite
tangle $L(L^1,\dots,L^n)\in\T(m_1+ \cdots + m_n)$ by taking the
relevant satellites of all components of $L$ (and reordering the
components of the resulting tangle appropriately).

The following theorem follows directly from the definition of the
operadic structure on $\T$ and the construction of the map $Z$
from tangles to diassociative trees  given by equation \eqref{eq:Z},
Section \ref{sub:dias}.
\begin{thm}\label{thm:operad}
The map $Z:\T\to\Dias$ is a morphism of operads.
\end{thm}

\end{document}